\documentclass[reqno, 12pt,a4letter]{amsart}
\usepackage{amsmath,amsxtra,amssymb,latexsym, amscd,amsthm}
\usepackage{graphicx,color}
\usepackage{subfigure}
\usepackage{tikz-cd}
\usepackage[utf8]{inputenc}
\usepackage{enumerate}
\usepackage[mathscr]{euscript}
\usepackage{mathrsfs}
\usepackage[english]{babel}
\usepackage{color}
\newtheorem {theorem}{Theorem}[section]
\newtheorem {corollary}[theorem]{Corollary}

\newtheorem {lemma}[theorem]{Lemma}
\newtheorem {example}[theorem]{Example}
\newtheorem {definition}[theorem]{Definition}
\newtheorem {remark}[theorem]{Remark}

\def\ar{a\kern-.370em\raise.16ex\hbox{\char95\kern-0.53ex\char'47}\kern.05em}
\def\ees{{\accent"5E e}\kern-.385em\raise.2ex\hbox{\char'23}\kern-.08em}
\def\eex{{\accent"5E e}\kern-.470em\raise.3ex\hbox{\char'176}}
\def\AR{A\kern-.46em\raise.80ex\hbox{\char95\kern-0.53ex\char'47}\kern.13em}
\def\EES{{\accent"5E E}\kern-.5em\raise.8ex\hbox{\char'23 }}
\def\EEX{{\accent"5E E}\kern-.60em\raise.9ex\hbox{\char'176}\kern.1em}
\def\ow{o\kern-.42em\raise.82ex\hbox{
  \vrule width .12em height .0ex depth .075ex \kern-0.16em \char'56}\kern-.07em}
\def\OW{O\kern-.460em\raise1.36ex\hbox{
\vrule width .13em height .0ex depth .075ex \kern-0.16em \char'56}\kern-.07em}
\def\UW{U\kern-.42em\raise1.36ex\hbox{
\vrule width .13em height .0ex depth .075ex \kern-0.16em \char'56}\kern-.07em}
\def\DD{D\kern-.7em\raise0.4ex\hbox{\char '55}\kern.33em}

\newcommand{\R}{{\mathbb R}}

\def\ees{{\accent"5E e}\kern-.385em\raise.2ex\hbox{\char'23}\kern-.08em}
\def\EES{{\accent"5E E}\kern-.5em\raise.8ex\hbox{\char'23 }}
\def\ow{o\kern-.42em\raise.82ex\hbox{
\vrule width .12em height .0ex depth .075ex \kern-0.16em \char'56}\kern-.07em}
\def\OW{O\kern-.460em\raise1.36ex\hbox{
\vrule width .13em height .0ex depth .075ex \kern-0.16em \char'56}\kern-.07em}

\title[]{Existence theorems for optimal solutions in semi-algebraic optimization}

\author{Jae Hyoung Lee$^\dagger$}
\address{Department of Applied Mathematics, Pukyong National University, Busan 48513, Korea}
\email{mc7558@naver.com}

\author{GUE MYUNG LEE$^\ddagger$}
\address{Department of Applied Mathematics, Pukyong National University, Busan 48513, Korea}
\email{gmlee@pknu.ac.kr}

\author{TI\EES N S\OW N PH\d{A}M$^{*}$}
\address{Department of Mathematics, Dalat University, 1 Phu Dong Thien Vuong, Dalat, Vietnam}
\email{sonpt@dlu.edu.vn}
\thanks{$^\ddagger$Corresponding author}

\date{ \today}

\begin{document}
\maketitle

\begin{abstract}
Consider the problem of minimizing a lower semi-continuous semi-algebraic function $f \colon \mathbb{R}^n \to \mathbb{R} \cup \{+\infty\}$ on an unbounded closed semi-algebraic set $S \subset \mathbb{R}^n.$ Employing adequate tools of semi-algebraic geometry, we first establish some properties of the tangency variety of the restriction of $f$ on $S.$ Then we derive verifiable necessary and sufficient conditions for the existence of optimal solutions of the problem as well as the boundedness from below and coercivity of the restriction of $f$ on $S.$ We also present a computable formula for the optimal value of the problem.
\end{abstract}

\section{Introduction}

The existence of optimal solutions for optimization problems has been an essential research topic in optimization theory.

It is well-known that a linear function attains its infimum on a nonempty polyhedral set if it is bounded from below on the set.

In 1956, Frank and Wolfe~\cite{Frank1956} proved that a quadratic function attains its infimum on a nonempty polyhedral set if it is bounded from below on the set. In 1982, Andronov, Belousov and Shironin~\cite{Andronov1982} (see also~\cite{Klatte2019}) showed that this result is still true if the quadratic objective function is replaced by a cubic function.

In 2002, Belousov and Klatte~\cite{Belousov2002} established the existence of optimal solutions for convex polynomial optimization problems.
In 2014, \DD inh, H\`a and Ph\d{a}m~\cite{Dinh2014-2} (see also~\cite{ PHAMTS2019-1}) proved that non-degenerate polynomial optimization problems have optimal solutions. Here, we should assume that for each considered problem, the objective function is bounded from below on the constraint set.

Very recently, for a general polynomial optimization problem, Ph\d{a}m~\cite{PHAMTS2023} provided necessary and sufficient conditions for the existence of optimal solutions of the problem as well as the boundedness from below and coercivity of the objective function on the constraint set, where the results are presented in terms of the tangency variety of the polynomials defining the problem. Since polynomials form a subclass of semi-algebraic functions, it is natural to extend these results for semi-algebraic optimization problems.

\subsection*{Contributions}
Given a lower semi-continuous semi-algebraic function $f \colon \mathbb{R}^n \rightarrow \mathbb{R} \cup \{+\infty\}$ and an unbounded closed semi-algebraic set $S \subset \mathbb{R}^n,$ we first establish some properties of the tangency variety of the restriction of $f$ on $S,$ and then we show that the following conditions can be characterized completely:
\begin{itemize}
\item $f$ is bounded from below on $S;$
\item $f$ attains its infimum on $S;$ and
\item $f$ is coercive on $S.$
\end{itemize}
Furthermore, we derive a computable formula for the optimal value $\inf_{x \in S} f(x).$

To be concrete, we study only semi-algebraic functions and sets. Analogous results, with essentially identical proofs, also hold for functions and sets definable in an o-minimal structure (see~\cite{Dries1996} for more on the subject). However, to lighten the exposition, we do not pursue this extension here.

The rest of this paper is organized as follows. Some definitions and preliminary results from variational analysis and semi-algebraic geometry are
recalled in Section~\ref{Section2}. Tangencies are introduced and studied in Section~\ref{Section3}. The main results are given in Section~\ref{Section4}. Finally, several examples are provided in Section~\ref{Section6}.

\section{Preliminaries} \label{Section2}

Throughout this work we shall consider the Euclidean vector space ${\Bbb R}^n$ endowed with its canonical scalar product $\langle \cdot, \cdot \rangle$ and we shall denote its associated norm $\| \cdot \|.$ The closed ball and the sphere centered at the origin $0 \in \mathbb{R}^n$ of radius $R > 0$ will be denoted by $\mathbb{B}_{R}$ and $\mathbb{S}_{R},$ respectively. We will adopt the convention that $\inf \emptyset = +\infty$ and $\sup \emptyset = -\infty.$

For a function $f \colon \mathbb{R}^n \rightarrow \mathbb{R} \cup \{+\infty\},$ we denote its {\em effective domain} and {\em epigraph} by,  respectively,
\begin{eqnarray*}
\mathrm{dom} f &:=& \{ x \in \mathbb{R}^n \ | \ f(x) < +\infty  \},\\
\mathrm{epi} f &:=& \{ (x, \alpha) \in \mathbb{R}^n \times \mathbb{R} \ | \ f(x) \le \alpha \}.
\end{eqnarray*}
The function $f$ is said to be {\em lower semi-continuous} if for each $x \in \mathbb{R}^n$ the inequality $\liminf_{x' \to {x}} f(x') \ge f({x})$ holds.
The {\em indicator function} of a set $S \subset \mathbb{R}^n,$ denoted $\delta_S,$ is defined by
$$\delta_S(x) :=
\begin{cases}
0 & \textrm{ if } x \in S, \\
+\infty & \textrm{ otherwise.}
\end{cases}$$

\subsection{Normals and subdifferentials}
Here we recall the notions of the normal cones to sets and the subdifferentials of real-valued functions used in this paper. The reader is referred to~\cite{Mordukhovich2006, Mordukhovich2018, Rockafellar1998} for more details.

\begin{definition}{\rm Consider a set $S\subset\mathbb{R}^n$ and a point ${x} \in S.$
\begin{enumerate}
\item[(i)] The {\em regular normal cone} (known also as the {\em prenormal} or {\em Fr\'echet normal cone}) $\widehat{N}({x}; S)$ to
$S$ at ${x}$ consists of all vectors $v\in\mathbb{R}^n$ satisfying
\begin{eqnarray*}
\langle v, x' - {x} \rangle &\le& o(\|x' -  {x}\|) \quad \textrm{ as } \quad x' \to {x} \quad \textrm{ with } \quad x' \in S.
\end{eqnarray*}

\item[(ii)] The {\em limiting normal cone} (known also as the {\em basic} or {\em Mordukhovich normal cone}) $N({x}; S)$ to $S$ at ${x}$ consists of all vectors $v \in \mathbb{R}^n$ such that there are sequences $x_k \to {x}$ with $x_k \in S$ and $v_k \rightarrow v$ with $v_k \in \widehat N(x_k; S).$
\end{enumerate}
}\end{definition}

If $S$ is a manifold of class $C^1,$ then for every point $x \in S,$ the normal cones $\widehat{N}({x}; S)$ and $N({x}; S)$ are equal to the normal space to $S$ at ${x}$ in the sense of differential geometry; see \cite[Example~6.8]{Rockafellar1998}. In particular, for all $t > 0$ and all $x \in \mathbb{S}_t,$ we have $N(x; \mathbb{S}_t) = \{\mu x \ | \ \mu \in \mathbb{R}\}.$

Functional counterparts of normal cones are subdifferentials.
\begin{definition}{\rm
Consider a function $f\colon\mathbb{R}^n \to \mathbb{R} \cup \{+\infty\}$ and a point ${x} \in \mathrm{dom} f.$ The {\em limiting} and {\em horizon subdifferentials} of $f$ at ${x}$ are defined respectively by
\begin{eqnarray*}
\partial f({x}) &:=& \big\{v\in\mathbb{R}^n\;\big|\;(v,-1)\in N\big(({x}, f({x}));\mathrm{epi} f\big)\big\}, \\
\partial^\infty f({x}) &:=& \big\{v\in\mathbb{R}^n\;\big|\;(v,0)\in N\big(({x}, f({x}));\mathrm{epi} f \big)\big\}.
\end{eqnarray*}
}\end{definition}

In \cite{Mordukhovich2006, Mordukhovich2018, Rockafellar1998} the reader can find equivalent analytic descriptions of the limiting subdifferential $\partial f({x})$ and comprehensive studies of it and related constructions. For convex $f,$ this subdifferential coincides with the convex subdifferential. Furthermore, if the function $f$ is of class $C^1,$ then $\partial f({x}) = \{\nabla f({x})\}.$ The horizon subdifferential $\partial^\infty f({x})$ plays an entirely different role--it detects horizontal ``normal'' to the epigraph--and it plays a decisive role in subdifferential calculus.

\begin{lemma} \label{Lemma23}
For any set $S \subset \mathbb{R}^n$ and point ${x} \in S,$ we have the representations
$$\partial \delta_S({x}) =  \partial^\infty \delta_S({x})  = N ({x}; S).$$
\end{lemma}
\begin{proof}
See \cite[Proposition~1.19]{Mordukhovich2018}.
\end{proof}

\begin{theorem}[Fermat rule]\label{FermatRule}
Consider a lower semi-continuous function $f \colon\mathbb{R}^n \to \mathbb{R} \cup \{+\infty\}$ and a closed subset $S$ of $\mathbb{R}^n.$
If ${x} \in \mathrm{dom} f \cap S$ is a local minimizer of $f$ on $S$ and the qualification condition
\begin{eqnarray*}
\partial^\infty f({x}) \cap \big( -N({x}; S)\big) &=& \{0\}
\end{eqnarray*}
is valid, then $0 \in \partial f({x}) + N({x}; S).$
\end{theorem}
\begin{proof}
This follows directly from \cite[Propositions~1.19~and~1.30 and Theorem~2.19]{Mordukhovich2018}.
\end{proof}

\subsection{Semi-algebraic geometry}

Now, we recall some notions and results of semi-algebraic geometry, which can be found in~\cite{Bochnak1998} and \cite[Chapter~1]{HaHV2017}.

\begin{definition}{\rm
A subset $S$ of $\mathbb{R}^n$ is called {\em semi-algebraic} if it is a finite union of sets of the form
$$\{x \in \mathbb{R}^n \ | \  f_i(x) = 0, \ i = 1, \ldots, p; f_i(x) > 0, \ i = p + 1, \ldots, q\},$$
where all $f_{i}$ are polynomials.
In other words, $S$ is a union of finitely many sets, each defined by finitely many polynomial equalities and inequalities.

A map $f \colon S \rightarrow {\mathbb{R}^m}$ is said to be {\em semi-algebraic} if its graph
\begin{eqnarray*}
\{ (x, y) \in S \times \mathbb{R}^m \ | \ y =  f(x) \}
\end{eqnarray*}
is a semi-algebraic set.
}\end{definition}


A major fact concerning the class of semi-algebraic sets is its stability under linear projections.

\begin{theorem}[Tarski--Seidenberg theorem] \label{TarskiSeidenbergTheorem}
The image of any semi-algebraic set $S \subset \mathbb{R}^n$ under a projection to any linear subspace of $\mathbb{R}^n$ is a semi-algebraic set.
\end{theorem}

\begin{remark}{\rm
As an immediate consequence of the Tarski--Seidenberg Theorem, we get semi-algebraicity of any set $\{ x \in A : \exists y \in B,  (x, y) \in C \},$  provided that $A ,  B,$  and $C$  are semi-algebraic sets in the corresponding spaces. Also, $\{ x \in A : \forall y \in B,  (x, y) \in C \}$ is a semi-algebraic set as its complement is the union of the complement of $A$  and the set $\{ x \in A : \exists y \in B,  (x, y) \not\in C \}.$ Thus, if we have a finite collection of semi-algebraic sets, then any set obtained from them with the help of a finite chain of quantifiers is also semi-algebraic.
}\end{remark}

\begin{definition}{\rm
Let $S, T$ and $T'$ be semi-algebraic sets, $T' \subset T,$ and let $f \colon S \to T$ be a continuous semi-algebraic map. A {\em semi-algebraic trivialization of $f$ over $T',$} with fibre $F,$ is a semi-algebraic homeomorphism $h \colon F \times T' \to f^{-1}(T'),$ such that $f \circ h$ is the projection map $F \times T' \to T', (x, t) \mapsto t.$ We say that the semi-algebraic trivialization $h$ is {\em compatible} with a subset
$S'$ of $S$ if there is a subset $F'$ of $F$ such that $h(F' \times T') = S' \cap f^{-1}(T').$
}\end{definition}

\begin{theorem}[Hardt's semi-algebraic triviality]\label{HardtTheorem}
Let $S, T$ be two semi-algebraic sets, $f \colon S \to T$ a continuous semi-algebraic map, $\{S_i\}_{i = 1, \ldots, p}$ a finite family of semi-algebraic subsets of $S.$ Then there exists a finite partition of $T$ into semi-algebraic sets $T = \cup_{j = 1}^q T_j$ and, for each $j$ with $f^{-1}(T_j)\neq\emptyset,$ a semi-algebraic trivialization $h_j \colon F_j \times T_j \to f^{-1}(T_j)$ of $f$ over $T_j,$ compatible with $S_i,$ for $i = 1, \ldots, p.$
\end{theorem}

The following well-known lemmas will be of great importance for us.

\begin{lemma} [monotonicity lemma] \label{MonotonicityLemma}
Let $f \colon (a, b) \rightarrow \mathbb{R}$ be a semi-algebraic function. Then there are finitely many points $a =: t_0 < t_1 < \cdots < t_p := b$ such that for each $i = 0, \ldots, p - 1,$ the restriction of $f$ to the interval $(t_i, t_{i + 1})$ is analytic, and either constant, or strictly increasing or strictly decreasing.
\end{lemma}

\begin{lemma}[growth dichotomy lemma] \label{GrowthDichotomyLemma}
Let $f \colon (0, \epsilon) \rightarrow {\mathbb R}$ be a semi-algebraic function with $f(t) \ne 0$ for all $t \in (0, \epsilon).$ Then there exist constants $a \ne 0$ and $\alpha \in {\mathbb Q}$ such that $f(t) = at^{\alpha} + o(t^{\alpha})$ as $t \to 0^+.$
\end{lemma}

\begin{lemma}[curve selection lemma at infinity]\label{CurveSelectionLemmaAtInfinity}
Let $S\subset \mathbb{R}^n$ be a semi-algebraic set, and let $f := (f_1, \ldots,f_m) \colon  \mathbb{R}^n \to \mathbb{R}^m$ be a semi-algebraic map. Assume that there exists a sequence $\{x_k\}_{{k} \ge 1} \subset S$ such that $\lim_{k \to +\infty} \| x_k\| = \infty$ and $\lim_{k \to +\infty} f(x_k)  = y \in(\overline{\mathbb{R}})^m,$ where $\overline{\mathbb{R}} := \mathbb{R} \cup \{\pm \infty\}.$ Then there exists an analytic semi-algebraic curve $\phi \colon (R, +\infty) \to \mathbb{R}^n$ such that $\phi(t) \in S$ for all $t > R, \lim_{t \to +\infty} \|\phi(t)\| = \infty$ and $\lim_{t \to +\infty} f(\phi(t)) = y.$
\end{lemma}

\begin{lemma}[path connectedness]\label{pathconnectedness}
The following statements hold.
\begin{enumerate}[{\rm (i)}]
\item Every semi-algebraic set has a finite number of connected components and each such component is semi-algebraic.

\item Every connected semi-algebraic set $S$ is semi-algebraically path connected: for every two points $x, y$ in $S,$ there exists a continuous semi-algebraic curve $\phi \colon [0, 1] \to \mathbb{R}^n$ lying in $S$ such that $\phi(0)  = x$ and $\phi(1) = y.$
\end{enumerate}
\end{lemma}

\begin{lemma}[piecewise continuity of semi-algebraic functions]\label{PiecewiseContinuity}
Given a semi-algebraic function $f \colon S \rightarrow \mathbb{R},$ where $S$ is a semi-algebraic subset of $\mathbb R^n$, there is a finite partition of $S$ into path connected semi-algebraic sets $C_1, \ldots, C_p,$ such that for each $i = 1, \ldots, p,$ the restriction of $f$ on $C_i$ is continuous.
\end{lemma}

As a consequence of the curve selection lemma at infinity, we have the following fact.
\begin{corollary}[{see \cite[Corollary~2.11]{PHAMTS2023}}] \label{Corollary210}
Let $S\subset \mathbb{R}^n$ be a semi-algebraic set. Then $S$ is unbounded if and only if there exists a real number $R > 0$ such that the set $S \cap \mathbb{S}_t$ is nonempty for all $t > R.$
\end{corollary}

We close this section with the following fact (see \cite[Lemma~2.4]{PHAMTS2020}).
\begin{lemma} \label{ChainRule}
Consider a lower semi-continuous and semi-algebraic function $f \colon \mathbb{R}^n \rightarrow \mathbb{R} \cup \{+\infty\}$ and a semi-algebraic curve $\phi \colon [a, b] \rightarrow {\rm dom}f.$
Then for all but finitely many ${{t}} \in [a, b],$ the maps $\phi$ and $f \circ \phi$ are analytic at ${{t}}$ and satisfy
\begin{eqnarray*}
v \in \partial f (\phi({{t}})) &\Longrightarrow& \langle v, \dot{\phi}({{t}}) \rangle \ = \ \frac{d}{dt}(f\circ\phi)(t), \\
v \in \partial^\infty f (\phi({{t}})) &\Longrightarrow& \langle v, \dot{\phi}({{t}}) \rangle \ = \ 0,
\end{eqnarray*}
where $\dot{\phi}({{t}}) := \frac{d}{dt}\phi(t).$
\end{lemma}

\section{Tangencies} \label{Section3}

In order to formulate and prove the main results of the paper, we need some notation and auxiliary results.
Throughout the paper, let $f\colon \R^n \to \mathbb{R} \cup \{+\infty\}$ be a lower semi-continuous and semi-algebraic function and let $S$ be a closed semi-algebraic subset of $\R^n$ such that the set $\mathrm{dom} f \cap  S$ is nonempty and unbounded.

\begin{definition}{\rm
By the {\em set of critical points of $f$ on $S$} we mean the set
$$\Sigma\left(f,S\right) := \left\{x \in \mathrm{dom} f \cap  S\ | \ 0\in\partial f(x)+ N(x;S)\right\}.$$
}\end{definition}

By the Tarski--Seidenberg theorem (Theorem~\ref{TarskiSeidenbergTheorem}), $\Sigma(f, S)$ is a semi-algebraic set. Moreover, we have

\begin{lemma}\label{BD32}
$f\left(\Sigma\left(f,S\right)\right)$ is a finite subset of $\mathbb{R}.$
\end{lemma}

\begin{proof}
In view of Lemmas~\ref{pathconnectedness} and \ref{PiecewiseContinuity}, there exists a finite partition of $\Sigma(f, S)$ into path connected semi-algebraic sets $C_1, \ldots, C_p$ such that for each $i = 1, \ldots, p,$ the restriction of $f$ on ${C_i}$ is continuous. We will show that the function $f$ is constant on $C_i.$ To this end, let $\phi \colon [0, 1] \to \mathbb{R}^n$ be a continuous semi-algebraic curve lying in $C_i.$ Then
the function $f \circ \phi$ is continuous; furthermore, by definition, for each $t \in[0, 1]$ there exists $\omega(t) \in\partial f(\phi(t)) \cap \big(-N(\phi(t); S) \big).$
Note that $(\delta_S \circ \phi)(t) = 0$ and $\partial \delta_S (\phi(t)) = N(\phi(t);S)$ (see Lemma~\ref{Lemma23}). By Lemma~\ref{ChainRule}, for all but finitely many $t\in [0, 1],$ the maps $\phi$ and $f \circ \phi$ are analytic at $t$ and satisfy
\begin{equation*}
\frac{d}{dt} (f \circ \phi)(t) = \langle \omega(t), \dot{\phi}(t)\rangle =
- \langle -\omega(t), \dot{\phi}(t)\rangle = - \frac{d}{dt} (\delta_S \circ \phi)(t)=0.
\end{equation*}
It follows that the function $f \circ \phi$ is constant. Finally, since the semi-algebraic set $C_i$ is path connected, any two points in $C_i$ can be joined (in $C_i$) by a continuous semi-algebraic curve (by Lemma~\ref{pathconnectedness}). Therefore, $f$ is constant on $C_i,$ which completes the proof.
\end{proof}

\begin{lemma}\label{BD33}
There exists a real number $R > 0$ such that for all $t > R$ and all $x \in S \cap \mathbb{S}_t$ we have
\begin{equation*}
N(x; {S}) \cap \big(-N(x; \mathbb{S}_t) \big) = \{0\}.
\end{equation*}
\end{lemma}

\begin{proof}
Suppose to the contrary that the lemma does not hold: there exist sequences $x_k \in S,$ with $\lim_{k \to +\infty} \|x_k\|  = +\infty,$ and $\mu_k \in\mathbb{R}\setminus\{0\}$ such that
$-\mu_k x_k \in N(x_k ;S).$
Applying the Curve Selection Lemma at infinity (Lemma~\ref{CurveSelectionLemmaAtInfinity}) for the semi-algebraic set
$$\{(x, \mu) \in \mathbb{R}^n \times \mathbb{R} \ | \ x \in S, \mu \ne 0, -\mu x \in N(x; S)\}$$
and the semi-algebraic map
$$\mathbb{R}^n \times \mathbb{R} \to \mathbb{R}, \quad (x, \mu)  \mapsto \|x\|,$$
we get an analytic semi-algebraic curve $(\phi, \mu)  \colon (R,+\infty) \to \R^n \times \R$ with $\lim_{t \to +\infty} \|\phi(t) \| = +\infty$ such that for all $t > R,$ we have
$$\phi(t)\in S, \quad \mu(t) \ne 0,  \quad \textrm{ and } \quad -\mu(t)\phi(t)\in N(\phi(t);S).$$
In view of Lemma~\ref{MonotonicityLemma},  we may assume that the function
$$(R, +\infty) \to \mathbb{R}, \quad t \mapsto \|\phi(t)\|^2,$$
is strictly increasing (perhaps after increasing $R$); in particular, $\frac{d}{d t}\|\phi(t)\|^2 > 0$ for all $t > R.$

On the other hand, since $(\delta_S \circ \phi)(t) = 0$ and  $\partial \delta_S(\phi(t)) = N(\phi(t); S),$ it follows from Lemma~\ref{ChainRule} that for all but finitely many $t \in (R, +\infty),$
\begin{equation*}
0 = \frac{d}{d t}(\delta_S \circ \phi)(t) = \langle - \mu(t)\phi(t), \dot{\phi}(t) \rangle  =  - \frac{\mu(t)}{2} \frac{d}{dt}\|\phi(t)\|^2.
\end{equation*}
Therefore, $\mu(t) = 0,$ a contradiction.
\end{proof}

\begin{lemma}\label{BD34}
There exists a real number $R > 0$ such that for all $t > R$ and all $x \in S \cap \mathbb{S}_t$ we have the inclusion
\begin{equation*}
N(x; S \cap \mathbb{S}_t) \subset N(x; S) + N(x; \mathbb{S}_t).
\end{equation*}
\end{lemma}

\begin{proof}
This follows directly from Lemma~\ref{BD33} and \cite[Theorem~2.16]{Mordukhovich2018}.
\end{proof}

\begin{definition}{\rm
We say that the {\em qualification condition} ($\mathrm{(QC)}$ for short) holds if
\begin{equation*}
\partial^\infty f(x)\cap \big( -N(x;S) \big) =\{0\} \quad \textrm{ for all } \quad x \in \mathrm{dom} f \cap  S.
\end{equation*}
We say that the {\em qualification condition at infinity} ($\mathrm{(QC)}_{\infty}$ for short) holds, if there exists $R > 0$ such that
\begin{equation*}
\partial^\infty f(x)\cap \big( -N(x;S) \big) =\{0\} \quad \textrm{ for all } \quad x \in (\mathrm{dom} f \cap  S) \setminus  \mathbb{B}_{R}.
\end{equation*}
}\end{definition}

Note that if $f$ is locally Lipschitz, then $\partial^\infty f(x) = \{0\}$ for all $x,$ and so the conditions $\mathrm{(QC)}$ and $\mathrm{(QC)}_{\infty}$ hold.

\begin{lemma}\label{BD36}
If $\mathrm{(QC)}_{\infty}$ holds, then there exists $R>0$ such that for all $t > R$ and all $x\in \mathrm{dom} f \cap  S \cap \mathbb{S}_t,$
\begin{equation*}
\partial^\infty f(x) \cap \big(-N(x;S\cap \mathbb{S}_t)\big) =\{0\}.
\end{equation*}
\end{lemma}

\begin{proof}
In view of Lemma~\ref{BD34}, it suffices to prove that there exists $R>0$ such that for all $t > R$ and all $x\in \mathrm{dom} f \cap  S \cap \mathbb{S}_t,$
\begin{equation*}
\partial^\infty f(x)\cap \big( -N(x;S)-N(x; \mathbb{S}_t) \big) =\{0\}.
\end{equation*}
By contradiction, there exist sequences $x_k \in \mathrm{dom} f \cap  S$ with $\lim_{k \to +\infty}\|x_k \|=+\infty,$ $\omega_k\in -N(x_k; S)$ and $\mu_k \in \mathbb{R}$ such that $\omega_k + \mu_k x_k \in \partial^\infty f(x_k) \setminus \{0\}.$  Applying the Curve Selection Lemma at infinity (Lemma~\ref{CurveSelectionLemmaAtInfinity}) for the semi-algebraic set
$$\{(x, \omega, \mu) \in \mathbb{R}^n \times \mathbb{R}^n \times \mathbb{R} \ | \ x \in \mathrm{dom} f \cap  S, \ \omega \in -N(x; S), \ \omega + \mu x \in
\partial^\infty f(x)\setminus\{0\} \}$$
and the semi-algebraic map
$$\mathbb{R}^n \times \mathbb{R}^n \times \mathbb{R} \to \mathbb{R}, \quad (x, \omega, \mu) \mapsto \|x\|,$$
we get an analytic semi-algebraic curve $(\phi, \omega, \mu) \colon (R, +\infty) \to\R^n \times \R^n \times \R$ such that
$\lim_{t \to +\infty} \|\phi(t) \| = +\infty$ and for all $t > R,$
\begin{align*}
&\phi(t) \in \mathrm{dom} f \cap  S, \ \omega(t) \in -N(\phi(t);S) \ \textrm{ and } \ \omega(t) + \mu(t)\phi(t)\in \partial^\infty f(\phi(t))\setminus\{0\}.
\end{align*}
In view of Lemma~\ref{MonotonicityLemma},  we may assume that the function $(R, +\infty) \to \mathbb{R}, t \mapsto \|\phi(t)\|,$ is strictly increasing (perhaps after increasing $R$); in particular, $\frac{d}{d t}\|\phi(t)\|^2 > 0$ for all $t > R.$ By Lemma~\ref{ChainRule}, we have for all but finitely many $t > R,$
\begin{eqnarray*}
0 &=& \langle \omega(t)+\mu(t)\phi(t), \dot{\phi}({{t}}) \rangle \\
   &=& \langle \omega(t), \dot{\phi}(t) \rangle + \frac{\mu(t)}{2}\frac{d}{d t}\|\phi(t)\|^2 \\
   &=& - \frac{d}{d t}(\delta_S \circ \phi)(t) + \frac{\mu(t)}{2}\frac{d}{d t}\|\phi(t)\|^2 \ = \ \frac{\mu(t)}{2}\frac{d}{d t}\|\phi(t)\|^2.
\end{eqnarray*}
Hence $\mu(t) = 0,$ and so $\omega(t) \in\partial^\infty f(\phi(t))\setminus\{0\},$ which contradicts our assumption that $\mathrm{(QC)}_{\infty}$  holds.
\end{proof}

\begin{definition}
{\rm By the {\em tangency variety of $f$ on $S$}, we mean the set
\begin{equation*}
\Gamma(f, S) :=  \{x \in \mathrm{dom} f \cap  S\ | \  \textrm{there exists $\mu \in \mathbb{R}$ such that } 0\in\partial f(x)+N(x;S)+\mu x\}.
\end{equation*}
}\end{definition}

Observe that $\Gamma(f, S)$ is a semi-algebraic set containing $\Sigma(f, S).$ Moreover, we have

\begin{lemma}\label{BD38}
Assume that $\mathrm{(QC)}_{\infty}$ holds. Then the tangency variety $\Gamma(f, S)$ is nonempty and unbounded.
\end{lemma}

\begin{proof}
Since the set $\mathrm{dom} f \cap  S$ is semi-algebraic and unbounded, it follows from Corollary~\ref{Corollary210} that
there exists a real number $R > 0$ such that for any $t > R,$ the set
$$\{x \in \mathrm{dom} f \cap  S \ | \ \|x\|^2 = t^2 \}$$
is nonempty and bounded. Hence, by the lower semi-continuity of $f,$ the optimization problem
\begin{equation*}
\textrm{minimize } \ f(x) \quad \textrm{ subject to } \quad x \in  S \  \textrm{ and } \  \|x\|^2 = t^2
\end{equation*}
has at least one optimal solution, say, $\phi(t).$ Clearly, $\phi(t) \in \mathrm{dom} f.$
By Theorem~\ref{FermatRule} and Lemmas~\ref{BD34}, \ref{BD36}, we have for all $t$ large enough,
\begin{eqnarray*}
0 \ \in \  \partial f(\phi(t))+N(\phi(t);S\cap \mathbb{S}_t) &\subset& \partial f(\phi(t))+N(\phi(t);S)+N(\phi(t);\mathbb{S}_t) \\
&=& \partial f(\phi(t))+N(\phi(t);S)+\{\mu\phi(t) \ | \ \mu\in\R\}.
\end{eqnarray*}
Therefore, $\phi(t)\in \Gamma(f,S).$ Since $\|\phi(t)\| = t,$ $\Gamma(f,S)$ is unbounded.
\end{proof}

By Lemma~\ref{PiecewiseContinuity}, there is a finite partition of $\Gamma(f, S)$ into semi-algebraic sets $C_i, i = 1, \ldots, \ell$ such that the restriction of $f$ on ${C_i}$ is continuous. Applying Hardt's triviality theorem (Theorem~\ref{HardtTheorem}) for the continuous semi-algebraic function
$$\rho \colon \Gamma(f, S) \rightarrow \mathbb{R}, \quad x \mapsto \|x\|,$$
we find a real number $R > 0,$ semi-algebraic sets $F_i, i = 1, \ldots, \ell$ and a semi-algebraic homeomorphism
$$h \colon (\cup_{i = 1}^{\ell} F_i) \times (R, +\infty) \rightarrow   \Gamma(f, S) \setminus \mathbb{B}_{R}$$
such that $h(F_i \times (R, +\infty)) = C_i \setminus \mathbb{B}_R$ for $i = 1, \ldots, \ell$ and the following diagram commutes:
\begin{equation*}
\begin{tikzcd}
\Gamma(f, S) \setminus \mathbb{B}_R \ar[r, "\rho"]
&  (R, +\infty) \\
(\cup_{i = 1}^{\ell} F_i) \times (R, +\infty)  \ar[u, "h"] \ar[ru, shift left, "\pi"]
\end{tikzcd}
\end{equation*}
where $\pi$ is the projection on the second component of the product, i.e., $\pi(x, t) = t.$ Since $F_i$ is semi-algebraic, the  number of its connected components, say, $p_i,$ is finite.
Then $C_i \setminus \mathbb{B}_R$ has exactly $p_i$ connected components, which are unbounded semi-algebraic sets. Therefore, we may decompose the set $\Gamma(f, S) \setminus \mathbb{B}_R $ as a disjoint union of finitely many semi-algebraic sets
$\Gamma_k, k = 1, \ldots, p := \sum_{i = 1}^\ell p_i$ such that the following conditions hold:
\begin{enumerate}[{\rm (i)}]
\item $\Gamma_k$ is connected and unbounded;
\item for each $t > R,$ the set $\Gamma_k \cap \mathbb{S}_t$ is nonempty and connected; and
\item the restriction of $f$ on $\Gamma_k$ is continuous.
\end{enumerate}
Corresponding to each $\Gamma_k,$ let $$f_k \colon (R, +\infty) \rightarrow \mathbb{R}, \quad t \mapsto f_k(t),$$
be the function defined by $f_k(t) :=  f(x),$ where $x \in \Gamma_k \cap \mathbb{S}_t.$ The definition is well-posed as shown in the following lemma.

\begin{lemma}\label{BD39}
Assume that $\mathrm{(QC)}_{\infty}$ holds. For all $R$ large enough and all $k = 1, \ldots, p,$ the following statements hold:
\begin{enumerate}
\item[{\rm (i)}]  The function $f_k$ is well-defined and semi-algebraic$;$
\item[{\rm (ii)}] The function $f_k$ is either constant or strictly monotone$;$
 \item[{\rm (iii)}] The function $f_k$ is constant if and only if $\Gamma_k \subset \Sigma(f, S).$
\end{enumerate}
\end{lemma}

\begin{proof}
(i) Fix $t > R$ and let $\phi\colon[0,1]\to\R^n$ be a continuous semi-algebraic curve such that for all $\tau\in[0, 1]$ we have $\phi(\tau)\in\Gamma_k\cap\mathbb{S}_t.$ By definition, $\|\phi(\tau)\| =t$ and there exist $\omega(\tau) \in -N(\phi(\tau);S)$  and $\mu(\tau)\in \R$ such that
\begin{eqnarray*}
 \omega(\tau)  + \mu(\tau)\phi(\tau) \in\partial f(\phi(\tau)).
\end{eqnarray*}
By Lemma~\ref{ChainRule}, for all but finitely many $\tau \in [0, 1],$ the maps $\phi$ and $f \circ \phi$ are analytic at $\tau$ and satisfy
\begin{eqnarray*}
\frac{d}{d \tau} (f \circ \phi)(\tau)
&=& \langle \omega(\tau) + \mu(\tau) \phi(\tau), \dot{\phi}(\tau)\rangle \\
&=& \langle \omega(\tau), \dot{\phi}(\tau) \rangle + \mu(\tau) \langle   \phi(\tau), \dot{\phi}(\tau)\rangle \\
&=& -\frac{d}{d \tau} (\delta_S \circ \phi)(\tau)  + \frac{\mu(\tau)}{2} \frac{d}{d \tau} \|\phi(\tau)\|^2 \\
&=& 0 + \frac{\mu(\tau)}{2} \frac{d}{d \tau} t^2 \ = \ 0.
\end{eqnarray*}
It follows that the function $f \circ \phi$ is constant. Now, since the semi-algebraic set $\Gamma_k \cap \mathbb{S}_t$ is path connected, any two points in $\Gamma_k \cap \mathbb{S}_t$ can be joined (in $\Gamma_k \cap \mathbb{S}_t$) by a continuous semi-algebraic curve (by Lemma~\ref{pathconnectedness}). Therefore, $f$ is constant on $\Gamma_k \cap \mathbb{S}_t,$ and so $f_k$ is well-defined.
The semi-algebraicity of $f_k$ is easy to check, and so is left to the reader.

(ii) Increasing $R$ (if necessary) and applying the monotonicity lemma (Lemma~\ref{MonotonicityLemma}), the claim follows.

(iii) {\em Necessity.}  By Hardt's triviality theorem (Theorem~\ref{HardtTheorem}) and by increasing $R$ (if necessary), we get a semi-algebraic set $F \subset \mathbb{R}^n$ and a semi-algebraic homeomorphism $\overline h \colon F \times (R, +\infty)\to \Gamma_k \setminus \Sigma(f, S)$ such that the following diagram commutes:
\begin{equation*}
\begin{tikzcd}
\Gamma_k\setminus \Sigma(f,S) \ar[r, "\rho"]
&  (R, +\infty) \\
F \times(R, +\infty)\ar[u, "\overline h"] \ar[ru, shift left, "\pi"]
\end{tikzcd}
\end{equation*}
where $\pi$ is the projection on the second component of the product, i.e., $\pi(x, t) = t.$

We now assume that the function $f_k$ is constant but $\Gamma_k \setminus \Sigma(f, S) \ne \emptyset.$ Let $x^*$ be any fixed point in $F$ and define the semi-algebraic function $\phi\colon (R,+\infty)\to \Gamma_k \setminus \Sigma(f, S)$ by $\phi(t)=\overline h(x^*,t)$ for all $t > R.$ The function $f|_{\Gamma_k}$ is constant so is $f\circ \phi.$ Since $\rho \circ \overline{h} = \pi,$ $\|\phi(t)\| = t$ for all $t > R.$ Since $\phi(t)\in\Gamma_k \setminus \Sigma(f, S),$ there exist $\omega(t) \in - N(\phi(t);S)$ and $\mu(t) \in \R$ such that
\begin{equation*}
\omega(t) + \mu(t)\phi(t) \in \partial f(\phi(t)) .
\end{equation*}
By Lemma~\ref{ChainRule}, for all but finitely many $t > R,$ the maps $\phi$ and $f \circ \phi$ are analytic at $t$ and satisfy
\begin{eqnarray*}
0 \ = \ \frac{d}{dt} (f \circ \phi)(t)
   &=& \langle  \omega(t) + \mu(t)\phi(t), \dot{\phi}({{t}}) \rangle \\
   &=& \langle \omega(t), \dot{\phi}(t) \rangle + \frac{\mu(t)}{2}\frac{d}{d t}\|\phi(t)\|^2 \\
   &=& - \frac{d}{d t}(\delta_S \circ \phi)(t) + \mu(t) t \ = \  \mu(t) t.
\end{eqnarray*}
Hence $\mu(t) = 0$ and so $\phi(t) \in \Sigma(f, S),$ a contradiction.

{\em Sufficiency.} Assume that $\Gamma_k \subset \Sigma(f, S).$ It suffices to show that the restriction of $f$ on $\Gamma_k$ is constant. To see this, let $\phi \colon [0, 1] \to \mathbb{R}^n$ be a continuous semi-algebraic curve lying in $\Gamma_k.$ We have $(f \circ \phi) ([0, 1]) \subset f(\Gamma_k) \subset f(\Sigma(f, S))$-a finite set in view of Lemma~\ref{BD32}. Hence, by Lemma~\ref{MonotonicityLemma}, there are finitely many points $0 = t_0 < t_1 < \cdots < t_\ell = 1$ such that the restriction of $f \circ \phi $ to the interval $(t_i, t_{i + 1}), i = 0, \ldots, \ell - 1,$ is constant. On the other hand, by construction, the restriction of $f$ on $\Gamma_k$ is continuous. Therefore, $f \circ \phi$ is constant. Finally, since the semi-algebraic set $\Gamma_k$ is path connected, any two points in $\Gamma_k$ can be joined (in $\Gamma_k$) by a continuous semi-algebraic curve (by Lemma~\ref{pathconnectedness}). Therefore, $f$ is constant on $\Gamma_k,$ which completes the proof.
\end{proof}

For any $t > R,$ the set $\mathrm{dom} f \cap  S \cap \mathbb{S}_t$ is nonempty and bounded. Since $f$ is lower semi-continuous and semi-algebraic, the function
$$\psi \colon (R, +\infty) \rightarrow \mathbb{R}, \quad t \mapsto \psi(t) := \min_{x \in S \cap \mathbb{S}_t} f(x),$$
is well-defined and semi-algebraic.
With this definition, we have the following three lemmas, whose proofs are similar to those in \cite[Lemmas~3.18, 3.19~and~3.20]{PHAMTS2023} and are included here for the sake of completeness.

\begin{lemma} \label{BD310}
Assume that $\mathrm{(QC)}_{\infty}$ holds. Then
for $R$ large enough, the following statements hold:
\begin{enumerate}
\item [{\rm (i)}] Any two of the functions $\psi, f_1, \ldots, f_p$ either coincide or are distinct.

\item [{\rm (ii)}] $\psi(t) = \min_{k = 1, \ldots, p} f_k(t)$ for all $t > R.$

\item [{\rm (iii)}] There is an index $k \in \{1, \ldots, p\}$ such that $\psi(t) = f_k(t)$ for all $t > R.$
\end{enumerate}
\end{lemma}

\begin{proof}
(i) This is an immediate consequence of the monotonicity lemma (Lemma~\ref{MonotonicityLemma}) and the semi-algebraicity of the functions in question.

(ii) By Corollary~\ref{Corollary210}, for all $t > R,$ $\Gamma(f,S)\cap \mathbb{S}_t\neq\emptyset,$ $\Gamma_k\cap \mathbb{S}_t\neq\emptyset,$ $k=1,\ldots,p,$ and $S\cap \mathbb{S}_t\neq\emptyset.$
Since $\Gamma(f,S)\subset S,$
\begin{eqnarray*}
\min_{x \in S \cap \mathbb{S}_t} f(x)\leq \min_{x \in \Gamma(f, S) \cap \mathbb{S}_t} f(x).
\end{eqnarray*}
Let $\tilde x$ be an optimal solution of the problem:
\begin{eqnarray*}
\min_{x \in S \cap \mathbb{S}_t} f(x).
\end{eqnarray*}
By Theorem~\ref{FermatRule}, Lemma~\ref{BD34}, and Lemma~\ref{BD36}, there exists $\mu\geq0$ such that
\begin{eqnarray*}
0\in\partial f(\tilde x)+N(\tilde x; S)+\mu\tilde x.
\end{eqnarray*}
Thus, $\tilde x\in \Gamma(f,S)\cap  \mathbb{S}_t,$  and so,
\begin{eqnarray*}
\min_{x \in S \cap \mathbb{S}_t} f(x)= \min_{x \in \Gamma(f, S) \cap \mathbb{S}_t} f(x).
\end{eqnarray*}
Moreover,
\begin{eqnarray*}
\min_{x \in \Gamma(f, S) \cap \mathbb{S}_t} f(x)  = \min_{x \in \left(\cup_{k=1}^p\Gamma_k\right) \cap \mathbb{S}_t} f(x)= \min_{k = 1, \ldots, p} \min_{x \in \Gamma_k \cap \mathbb{S}_t} f(x)  =  \min_{k = 1, \ldots, p}  f_k(t).
\end{eqnarray*}
Hence $\psi(t) = \min_{k = 1, \ldots, p} f_k(t)$ for all $t > R.$

(iii) This follows from items (i) and (ii).
\end{proof}

We have associated to the function $f$ a finite number of functions $f_k$ of a single variable,
each function $f_k$ is either constant or strictly monotone. In particular, the following limits exist:
$$\lambda_k := \lim_{t \to +\infty} f_k(t) \in \mathbb{R} \cup \{\pm \infty\} \quad \textrm{ for } \quad k = 1, \ldots, p.$$
In view of Lemma~\ref{BD39}(iii), if $f_k \equiv \lambda_k,$ then $\lambda_k \in f(\Sigma(f, S)).$ Furthermore, by Lemma~\ref{BD310}(iii), the limit $\lim_{t \to +\infty} \psi(t)$ exists and equals to $\lambda_k$ for some $k.$

\begin{lemma} \label{BD311}
Assume that $\mathrm{(QC)}_{\infty}$ holds. Then
\begin{eqnarray*}
\lim_{t \to +\infty} \psi(t)  &=&  \min_{k = 1, \ldots, p} \lambda_k.
\end{eqnarray*}
\end{lemma}

\begin{proof}
Indeed, by Lemma~\ref{BD310}(ii), $\psi(t) \le f_k(t)$ for all $t > R$ and all $k = 1, \ldots, p.$ Letting $t \to +\infty,$ we get
\begin{eqnarray*}
\lim_{t \to +\infty} \psi(t)  &\le& \min_{k = 1, \ldots, p} \lambda_k.
\end{eqnarray*}
On the other hand, by Lemma~\ref{BD310}(iii), there exists an index $k \in \{1, \ldots, p\}$ such that $\psi \equiv f_k,$ and so
\begin{eqnarray*}
\lim_{t \to +\infty} \psi(t) = \lambda_k.
\end{eqnarray*}
Combining this with the previous inequality, we get the desired conclusion.
\end{proof}

\begin{lemma} \label{BD312}
We have
\begin{eqnarray*}
\lim_{t \to +\infty} \psi(t)  &\ge& \inf_{x \in S} f(x)
\end{eqnarray*}
with the equality if $f$ does not attain its infimum on $S.$
\end{lemma}

\begin{proof}
Indeed, we have for all $t > R,$
\begin{eqnarray*}
\psi(t)  &=& \min_{x \in S \cap \mathbb{S}_t} f(x) \ \ge \ \inf_{x \in S} f(x).
\end{eqnarray*}
Letting $t \to +\infty,$ we get $\lim_{t \to +\infty} \psi(t) \ge \inf_{x \in S} f(x).$

Now suppose that $f$ does not attain its infimum on $S;$ then there exists a sequence $x_k \in \mathrm{dom} f \cap S$ such that
$$\lim_{k \to +\infty} \|x_k\| = +\infty \quad \textrm{ and } \quad \lim_{k \to +\infty} f(x_k) = \inf_{x \in S} f(x).$$
On the other hand, by definition, it is clear that $\psi(\|x_k\|) \le f(x_k)$ for all $k$ large enough. Therefore,
$\lim_{t \to +\infty} \psi(t) \le \inf_{x \in S} f(x),$ and so the desired conclusion follows.
\end{proof}

Note that in the above lemma we do not assume that $f$ is bounded from below on $S.$

\section{Results} \label{Section4}

Let $f\colon \R^n \to \mathbb{R} \cup \{+\infty\}$ be a lower semi-continuous and semi-algebraic function and let $S$ be a closed semi-algebraic subset of $\R^n$ such that the set $\mathrm{dom} f \cap  S$ is nonempty and unbounded. Consider the constrained optimization problem:
\begin{equation}\label{Problem}
\textrm{minimize } \ f(x) \quad \textrm{ subject to } \quad x \in S. \tag{P}
\end{equation}
Following the approach in \cite{PHAMTS2023}, we provide verifiable necessary and sufficient conditions for the existence of optimal solutions of the problem~\eqref{Problem} as well as the boundedness from below and coercivity of the restriction of $f$ on $S.$ We also present a computable formula for the optimal value of the problem.

Keeping the notation as in the previous section, we can write $\Gamma(f, S) \setminus \mathbb{B}_R = \cup_{k = 1}^p \Gamma_k,$ where
each $\Gamma_k$ is an unbounded connected semi-algebraic set. Corresponding to each $\Gamma_k,$ the semi-algebraic functions
$$f_k \colon (R, +\infty) \rightarrow \mathbb{R}, \quad t \mapsto f_k(t) := f|_{\Gamma_k \cap \mathbb{S}_t},$$
are well-defined, and so are the real numbers
$$\lambda_k := \lim_{t \to +\infty} f_k(t) \in \mathbb{R} \cup \{\pm \infty\}.$$
Also, recall that the semi-algebraic function $\psi \colon (R, +\infty) \rightarrow \mathbb{R}$ is defined by
$$\psi(t) := \min_{x \in S \cap \mathbb{S}_t} f(x).$$
Here and in the following, $R$ is chosen large enough so that the conclusions of Lemmas~\ref{BD33}, \ref{BD34}, \ref{BD36}, \ref{BD39} and \ref{BD310} hold.

\subsection{Boundedness from below}

In this subsection we present necessary and sufficient conditions for the boundedness from below of the objective function $f$ on the constraint set $S.$

\begin{theorem} \label{Theorem41}
Assume that $\mathrm{(QC)}_{\infty}$ holds. Then $f$ is bounded from below on $S$ if and only if it holds that
$$\min_{k = 1, \ldots, p} \lambda_k  >  -\infty.$$
\end{theorem}

\begin{proof}
From the proof of Lemma~\ref{BD310},
\begin{eqnarray*}
\psi(t) = \inf_{x \in S\cap \mathbb{S}_t} f(x)=\min_{k = 1, \ldots, p} f_k(t).
\end{eqnarray*}
So, if $f$ is bounded from below on $S,$ for $k=1,\ldots,p,$
\begin{eqnarray*}
-\infty<\inf_{x\in S}f(x)\leq \inf_{x \in S\cap \mathbb{S}_t} f(x)\leq f_k(t),
\end{eqnarray*}
and so, $ -\infty<\min_{k = 1, \ldots, p} \lambda_k.$

Suppose that $f$ is not bounded from below on $S.$
Then there exist $x_l\in S,$ $l=1,2,\ldots,$ such that $f(x_l)\to-\infty$ as $l\to\infty.$
We may assume that $\|x_l\|\to+\infty$ as $l\to\infty.$
Then
\begin{eqnarray*}
-\infty=\inf_{x\in S}f(x)=\lim_{l\to\infty}f(x_l)\geq \lim_{l\to\infty}\inf_{x\in S\cap \mathbb{S}_{\|x_l\|}}f(x)=\lim_{t\to\infty}\psi(t)\geq\min_{k = 1, \ldots, p} \lambda_k,
\end{eqnarray*}
where the last equality holds since $\lim_{t\to\infty}\psi(t)$ exists (when we include $\pm\infty$ as the limits) and the last inequality holds from Lemma~\ref{BD310} (iii).
Hence
$$\min_{k = 1, \ldots, p} \lambda_k  =  -\infty.$$
\end{proof}

In what follows we let
$$K :=\{k \ | \ f_k \textrm{ is not constant} \}.$$
By the growth dichotomy lemma (Lemma~\ref{GrowthDichotomyLemma}), we can assume that each function $f_k, k \in K,$ is developed into a fractional power series of the form
\begin{eqnarray*}
f_k(t) &=& a_k t^{\alpha_k} + \textrm{ lower order terms in } t \quad \textrm{ as } \quad  t \to +\infty,
\end{eqnarray*}
where $a_k \in \mathbb{R} \setminus \{0\}$ and $\alpha_k \in \mathbb{Q}.$

\begin{theorem} \label{Theorem42}
Assume that $\mathrm{(QC)}_{\infty}$ holds. Then $f$ is bounded from below on $S$ if and only if for any $k \in K,$
$$\alpha_k > 0 \quad \Longrightarrow \quad a_k > 0.$$
\end{theorem}

\begin{proof}
In light of Lemma~\ref{BD32}, $f(\Sigma(f, S))$ is a finite subset of $\mathbb{R}.$ By Lemma~\ref{BD39}(iii), if $k \not \in K,$ then $f|_{\Gamma_k} \equiv \lambda_k$ and $\Gamma_k \subset \Sigma(f, S),$ which yield $\lambda_k \in f(\Sigma(f, S)),$ and so $\lambda_k$ is finite. Therefore, in view of Theorem~\ref{Theorem41}, $f$ is bounded from below on $S$ if and only if it holds that
$$\lambda_k = \lim_{t \to +\infty} f_k(t) >  -\infty \quad \textrm{ for all } \quad k \in K.$$
Then the desired conclusion follows immediately from the definition of $\alpha_k$ and $a_k.$
\end{proof}

\subsection{Optimal values}
The following result shows that to compute the optimal value of the problem~\eqref{Problem} it suffices to know the {\em finite} set $f(\Sigma(f, S))$ and the values $\lambda_k, k = 1, \ldots, p.$

\begin{theorem}\label{Theorem43}
Assume that $\mathrm{(QC)}$ holds. Then
\begin{eqnarray*}
\inf_{x \in S} f(x) & = & \min \left \{ \min_{x \in \Sigma(f, S)} f(x), \min_{k = 1, \ldots, p} \lambda_k \right \}.
\end{eqnarray*}
\end{theorem}

\begin{proof}
We first assume that $f$ attains its infimum on $S,$ i.e., there exists a point $x^* \in S$ such that
\begin{eqnarray*}
f(x^*) &=& \inf_{x \in S} f(x).
\end{eqnarray*}
In light of Theorem~\ref{FermatRule}, $x^* \in \Sigma(f, S)$ and so
\begin{eqnarray*}
\inf_{x \in S} f(x) & \ge & \min_{x \in \Sigma(f, S)} f(x).
\end{eqnarray*}

We now assume that $f$ does not attain its infimum on $S.$  Then there exists a sequence $x_k \in \mathrm{dom} f \cap S$ such that
\begin{eqnarray*}
\lim_{k \to +\infty} \|x_k\| = +\infty  \quad & \textrm{ and } & \quad \lim_{k \to +\infty} f(x_k)  \ = \ \inf_{x \in S} f(x).
\end{eqnarray*}
Since the function $f$ is lower semi-continuous and the set $\{x \in \mathrm{dom} f \cap S \ | \ \|x\|^2 = \|x_k\|^2 \}$ is nonempty compact, the optimization problem
\begin{eqnarray*}
\textrm{minimize } \ f(x) \quad \textrm{ subject to } \quad x \in  S \  \textrm{ and } \  \|x\|^2 = \|x_k\|^2
\end{eqnarray*}
has at least one optimal solution, say, $y_k.$ Clearly, $y_k$ belongs to $\mathrm{dom} f \cap S$ and satisfies
\begin{eqnarray*}
\lim_{k \to +\infty} \|y_k\| = +\infty  \quad & \textrm{ and } & \quad \lim_{k \to +\infty} f(y_k)  \ = \ \inf_{x \in S} f(x).
\end{eqnarray*}
Furthermore, by Theorem~\ref{FermatRule} and Lemmas~\ref{BD34}, \ref{BD36}, we have for all $k$ large enough,
\begin{eqnarray*}
0 \ \in \ \partial f(y_k) + N(y_k; S \cap \mathbb{S}_{\|x_k\|}) &\subset& \partial f(y_k) + N(y_k; S) + N(y_k; \mathbb{S}_{\|x_k\|}) \\
&=& \partial f(y_k) + N(y_k; S) + \{\mu y_k \ | \ \mu \in \mathbb{R}\},
\end{eqnarray*}
and so $y_k \in \Gamma(f, S).$ Passing to a subsequence if necessary, we may assume that for all $k,$ $y_k \in \Gamma_{\ell}$ for some $\ell \in \{1, \ldots, p\}.$ Then
\begin{eqnarray*}
\inf_{x \in S} f(x) = \lim_{k \to +\infty} f(y_k)  = \lim_{k \to +\infty}  f_{\ell}(\|y_k\|) = \lambda_{\ell}.
\end{eqnarray*}

Therefore, in both cases, we have
\begin{eqnarray*}
\inf_{x \in S} f(x) \geq \min \left \{ \min_{x \in \Sigma(f, S)} f(x), \min_{k = 1, \ldots, p} \lambda_k \right \}.
\end{eqnarray*}

Since $\Sigma(f,S)\subset S,$
\begin{eqnarray*}
\inf_{x \in S} f(x)\leq \min_{x \in \Sigma(f, S)} f(x).
\end{eqnarray*}
Since $\Gamma_k\cap \mathbb{S}_t\subset S,$ $k=1,\ldots,p,$
\begin{eqnarray*}
\inf_{x \in S} f(x)\leq f_k(t), \ \ k=1,\ldots,p,
\end{eqnarray*}
and hence,
\begin{eqnarray*}
\inf_{x \in S} f(x)\leq \min_{k=1,\ldots,p}f_k(t).
\end{eqnarray*}
Thus,
\begin{eqnarray*}
\inf_{x \in S} f(x) \leq \min \left \{ \min_{x \in \Sigma(f, S)} f(x), \min_{k = 1, \ldots, p} \lambda_k \right \}.
\end{eqnarray*}

\end{proof}

\subsection{Existence of optimal solutions}

In this subsection we provide necessary and sufficient conditions for the existence of optimal solutions to the problem~\eqref{Problem}.
We start with the following result.

\begin{theorem} \label{Theorem44}
Assume that $\mathrm{(QC)}$ holds. Then $f$ attains its infimum on $S$ if and only if it holds that
\begin{eqnarray*}
\Sigma(f, S) \ne \emptyset \quad \textrm{ and } \quad \min_{x \in \Sigma(f, S) } f(x) & \le & \min_{k \in K} \lambda_k.
\end{eqnarray*}
\end{theorem}

\begin{proof}
Note that $f(\Sigma(f, S))$ is a finite subset of $\mathbb{R}$ (see Lemma~\ref{BD32}).

{\em Necessity.}  Let $f$ attain its infimum on $S,$ i.e., there exists a point $x^* \in S$ such that
\begin{eqnarray*}
f(x^*) &=& \inf_{x \in S} f(x).
\end{eqnarray*}
In light of Theorem~\ref{FermatRule}, $x^* \in \Sigma(f, S)$ and so $\Sigma(f, S)$ is nonempty. Moreover, we have
\begin{eqnarray*}
\min_{x \in \Sigma(f, S) } f(x) & \le &
f(x^*) \ = \ \inf_{x \in S} f(x) \ \le \
\min_{k = 1, \ldots, p} \lambda_k \ \le \ \min_{k \in K} \lambda_k,
\end{eqnarray*}
where the second inequality follows from Theorem~\ref{Theorem43}.

{\em Sufficiency.}  By the assumption, we have
\begin{eqnarray*}
- \infty & < & \min_{x \in \Sigma(f, S)} f(x) \ \le \ \min_{k \in K} \lambda_k.
\end{eqnarray*}
On the other hand, it is clear from Lemma~\ref{BD39}(iii) that $\lambda_k \in f(\Sigma(f, S))$ for all $k \not \in K$ and so
\begin{eqnarray*}
\min_{\lambda \in f(\Sigma(f, S))} \lambda & \le & \min_{k \not \in K} \lambda_k.
\end{eqnarray*}
Therefore,
\begin{eqnarray*}
- \infty & < & \min_{x \in \Sigma(f, S)} f(x) \ = \  \min_{\lambda \in f(\Sigma(f, S))} \lambda  \ \le \ \min_{k = 1, \ldots, p} \lambda_k,
\end{eqnarray*}
which, together with Theorem~\ref{Theorem41}, yields that $f$ is bounded from below on $S.$ Moreover, by Theorem~\ref{Theorem41}, we have
\begin{eqnarray*}
\min_{x \in \Sigma(f, S) } f(x) & = & \inf_{x \in S} f(x),
\end{eqnarray*}
which implies that $f$ attains its infimum on $S.$
\end{proof}

\begin{theorem}\label{Theorem45}
Assume that $\mathrm{(QC)}$ holds. Then the set of all optimal solutions of the problem~\eqref{Problem} is nonempty compact if and only if it holds that
\begin{eqnarray*}
\Sigma(f, S) \ne \emptyset, \quad \min_{x \in \Sigma(f, S) } f(x) & \le & \min_{k \in K} \lambda_k,
\quad \textrm{ and } \quad \min_{x \in \Sigma(f, S) } f(x) \  < \  \min_{k \not \in K} \lambda_k.
\end{eqnarray*}
\end{theorem}

\begin{proof}
Recall that, by Theorem~\ref{FermatRule}, if the problem~\eqref{Problem} has an optimal solution, then
\begin{eqnarray*}
 \min_{x \in \Sigma(f, S) } f(x) &=&  \inf_{x \in S} f(x).
\end{eqnarray*}

{\em Necessity.}
Take any $k \not \in K.$ Then $f|_{\Gamma_k} \equiv \lambda_k.$
It is clear that $ \inf_{x \in S} f(x) \leq \lambda_k.$
If $ \inf_{x \in S} f(x) = \lambda_k,$ then the set of all optimal solutions of the problem~\eqref{Problem} is unbounded.
Since $f|_{\Gamma_k}=\lambda_k$ and $\Gamma_k$ is unbounded,
\begin{eqnarray*}
\inf_{x \in S} f(x) < \lambda_k,
\end{eqnarray*}
which is contradicts the assumption.
Since $\Gamma_k$ is unbounded, our assumption implies that
\begin{eqnarray*}
\min_{x \in \Sigma(f, S) } f(x) &  = &  \inf_{x \in S} f(x) \ < \ \lambda_k,
\end{eqnarray*}
which, together with Theorem~\ref{Theorem44}, yields the desired conclusion.

{\em Sufficiency.} The function $f$ is lower semi-continuous. Hence, in view of Theorem~\ref{Theorem44}, it suffices to show that the set of all optimal solutions of the problem~\eqref{Problem} is bounded. Suppose to the contrary that the semi-algebraic set
$$\{x \in S \setminus \mathbb{B}_R \ | \ f(x) =  \inf_{x \in S} f(x)\}$$
is unbounded. By Lemma~\ref{pathconnectedness}(i), this set must contain an unbounded (semi-algebraic) connected component, say, $X.$
Observe that
$$X \subset \Sigma(f, S)\subset \Gamma(f, S).$$
Therefore, $X \subset \Gamma_k$ for some $k \in \{1, \ldots, p\}.$ Thanks to Corollary~\ref{Corollary210}, for all $t$ large enough, the set $X \cap \mathbb{S}_t$ is nonempty and so
\begin{eqnarray*}
f_k(t) &=&  f|_{\Gamma_k \cap \mathbb{S}_t} \ = \  f|_{X \cap \mathbb{S}_t}.
\end{eqnarray*}
Consequently, $f_k$ is constant $\inf_{x \in S} f(x),$ which yields $k \not \in K$ and $\lambda_k = \inf_{x \in S} f(x),$
which contradicts the assumption that
\begin{eqnarray*}
\inf_{x \in S} f(x) \ = \ \min_{x \in \Sigma(f, S)} f(x) &  < &  \min_{k\neq K}\lambda_k.
\end{eqnarray*}
The theorem is proved.
\end{proof}

\subsection{Coercivity}

The function $f$ is {\em coercive on} the set $S$ if for every sequence $x_k \in S$ such that $\|x_k\| \to +\infty,$ we have $f(x_k) \to +\infty.$ It is well known that if $f$ is coercive on $S,$ then $f$ achieves its infimum on $S.$ A necessary and sufficient condition for the coercivity of $f$ on $S$ is as follows.

\begin{theorem}\label{Theorem46}
Assume that $\mathrm{(QC)}_{\infty}$ holds. Then the following statements are equivalent:
\begin{enumerate}
\item[{\rm (i)}] The function $f$ is coercive on $S.$
\item[{\rm (ii)}] $\lambda_k = + \infty$ for all $k = 1, \ldots, p.$
\end{enumerate}
\end{theorem}

\begin{proof}
Recall that $\psi(t) := \min_{x \in S \cap \mathbb{S}_t} f(x)$ for $t > R.$ Hence, $f$ is coercive on $S$ if and only if  $\lim_{t \to +\infty} \psi(t) = +\infty,$ or equivalently, $\min_{k = 1, \ldots, p} \lambda_k = + \infty$ in view of Lemma~\ref{BD311}.
\end{proof}

\section{Examples} \label{Section6}

In this section we give examples to illustrate our main results.

\begin{example}{\rm
Let $S := \mathbb{R}^2$ and  $f(x,y):=x^2+|y|.$ A direct calculation shows that $N((x,y); \mathbb{R}^2) = \{(0,0)\},$
$\partial^{\infty} f(x, y) = \{(0, 0)\}$ (as $f$ is locally Lipschitz) and that
\begin{eqnarray*}
\partial f(x, y) = \left\{
                    \begin{array}{ll}
                      \{(2x, \xi) \ | \ \xi \in[-1,1]\} & \text{if } y=0, \\
                      \{(2x,1) \} & \text{if } y>0, \\
                      \{(2x,-1) \} & \text{if } y<0.
                    \end{array}\right.
\end{eqnarray*}
It follows that
$\Sigma(f, \mathbb{R}^2) = \{(0, 0)\}$ and
\begin{eqnarray*}
\Gamma(f, \mathbb{R}^2) = [\R\times\{0\}] \cup [\{0\}\times \R\setminus\{0\}] \cup \{ (x,y) \ | \ x\in\R, \ y=\pm\frac{1}{2}\}.
\end{eqnarray*}
Hence, for $R > \frac{1}{2},$ the set $\Gamma(f, \mathbb{R}^2) \setminus \mathbb{B}_R$ has eight connected components:
\begin{eqnarray*}
\Gamma_{\pm 1} &:=& \left\{ (\pm t, 0) \ | \ t > R \right\}, \\
\Gamma_{\pm 2} &:=& \left\{ (0, \pm t) \ | \ t > R \right\}, \\
\Gamma_{\pm 3} &:=& \left\{ \left(t, \frac{1}{2}\right) \ | \ t > \sqrt{R-\frac{1}{4}} \right\}, \\
\Gamma_{\pm 4} &:=& \left\{ \left(t, -\frac{1}{2}\right) \ | \ t > \sqrt{R^2-\frac{1}{4}} \right\}.
\end{eqnarray*}
Consequently, the restriction of $f$ on these components are given by
\begin{eqnarray*}
f|_{\Gamma_{\pm 1}} &=& t^2,\quad f|_{\Gamma_{\pm 2}} \ = \ t,\\
f|_{\Gamma_{\pm 3}} &=& f|_{\Gamma_{\pm 4}} \ = \  t^2+\frac{1}{4}.
\end{eqnarray*}
Thus
\begin{eqnarray*}
\lambda_{\pm 1} &=& \lambda_{\pm  2}  = \ \lambda_{\pm  3}  \ = \ \lambda_{\pm  4}  \ = \ +\infty.
\end{eqnarray*}
The results presented in the previous section show that the set of global minimizers of $f$ on $S$ is nonempty compact and that
\begin{eqnarray*}
\inf_{(x, y) \in \mathbb{R}^2} f(x, y) &=& \min_{(x, y) \in \Sigma(f, S)} f(x, y)  \ = \ f(0, 0) \ = \ 0.
\end{eqnarray*}
Furthermore, in light of Theorem~\ref{Theorem46}, $f$ is coercive.
}\end{example}

\begin{example}{\rm
Let $S := \mathbb{R}^2$ and  $f(x,y):=x+y.$ Then, by simple calculations, we have
\begin{eqnarray*}
\Gamma(f, \mathbb{R}^2) = \{ (x,y) \ | \ x=y\}.
\end{eqnarray*}
For $R>0,$ let $\Gamma_1:=\{(t,t) \ | \ t\geq R\}$ and let $\Gamma_2:=\{(-t,-t) \ | \ t\geq R\}.$
Then we see that the restriction of $f$ on these components are given by
\begin{eqnarray*}
f|_{\Gamma_{1}} = \sqrt{2t}, \ \ f|_{\Gamma_{2}}   -\sqrt{2t}.
\end{eqnarray*}
So, we have
\begin{eqnarray*}
\lambda_{1} =\lim_{t\to\infty}f|_{\Gamma_{1}}=+\infty, \ \ \lambda_{2} =\lim_{t\to\infty}f|_{\Gamma_{2}}=-\infty.
\end{eqnarray*}
and thus, by Theorem~\ref{Theorem41}, $f$ is not bounded from below on $S.$
}\end{example}

\begin{example}{\rm
Let $S := \mathbb{R}^2$  and $f(x, y) := (xy - 1)^2 + |y|.$ We have $N((x,y); \mathbb{R}^2) = \{(0,0)\},$
$\partial^{\infty} f(x, y) = \{(0, 0)\}$ (as $f$ is locally Lipschitz) and
\begin{eqnarray*}
\partial f(x,y) = \left\{
                    \begin{array}{ll}
                      \{(0,-2x+\xi) \ | \ \xi\in[-1,1]\} & \text{if } y=0, \\
                      \{(2(xy-1)y,2(xy-1)x+1) \}, & \text{if } y>0, \\
                      \{(2(xy-1)y,2(xy-1)x-1) \}, & \text{if } y<0.
                    \end{array}\right.
\end{eqnarray*}
It follows that
$\Sigma(f, \mathbb{R}^2) = [-\frac{1}{2}, \frac{1}{2}] \times \{0\}$ and
$$\Gamma(f, \mathbb{R}^2) = \Sigma(f, \mathbb{R}^2) \cup \{(x, y) \ | \ g_+(x, y) = 0, y > 0\} \cup \{(x, y) \ | \ g_-(x, y) = 0, y < 0\},$$
where $g_{\pm} (x, y) := -2\,{x}^{3}y+2\,x{y}^{3} \mp  x+2\,{x}^{2}-2\,{y}^{2}.$ Then we can see that\footnote{The computations are performed with the software Maple, using the command ``puiseux'' of the package ``algcurves'' for the rational Puiseux expansions.} for $R$ large enough, the set $\Gamma(f, \mathbb{R}^2) \setminus \mathbb{B}_R$ has eight connected components:
$$\begin{array}{llll}
\Gamma_{\sigma, 1}: &&
x := (-{t}^{-1}- \frac{1}{2}\,\sigma\,{t}^{2}+O \left({t}^{4} \right) ), \quad&
y := (-{t}^{-1}- \frac{1}{4}\,\sigma\,{t}^{2}+O \left( {t}^{4} \right) ),  \\
\Gamma_{\sigma, 2}: &&
x := ({\frac {1}{3}}{t}^{-1}+ \frac{3}{2}\,\sigma\,{t}^{2}+O \left( {t}^{4} \right)), \quad
& y := (-{\frac {1}{3}}{t}^{-1}+ \frac{3}{4}\,\sigma\,{t}^{2}+O \left( {t}^{4}
 \right) ), \\
\Gamma_{\sigma, 3}: &&
x := (-2\,t +4\,{t}^{3}+O \left( {t}^{4} \right) ), \quad
& y := (-{\frac {1}{2}}{t}^{-1}- t +2\,{t}^{3}+O \left( {t}^{4} \right) ), \\
\Gamma_{\sigma, 4}: &&
x := ({t}^{-1}+2\,t -\sigma\,{t}^{2}-4\,{t}^{3}+O \left( {t}^{4} \right) ), \quad
& y := (t - \frac{1}{2}\,\sigma\,{t}^{2}-2\,{t}^{3}+3\,\sigma\,{t}^{4}), \\
\end{array}$$
where $\sigma = \pm 1$ and $t \to \mp 0$ for $k = 1, 2, 3,$ and $t \to \pm 0$ for $k = 4.$ Then substituting these expansions in $f$ we get
$$\begin{array}{llll}
f|_{\Gamma_{\sigma, 1}} &=& ({t}^{-4}-2\,{t}^{-2} +  \frac{1}{2}\,\sigma\,{t}^{-1}+1+O \left({t} \right) ), \\
f|_{\Gamma_{\sigma, 2}} &=& ({\frac {1}{81}}{t}^{-4}+{\frac {2}{9}}{t}^{-2}-{\frac {5}{18}}\,
\sigma\,{t}^{-1}+1+O \left( t \right) ), \\
f|_{\Gamma_{\sigma, 3}} &=& (-\frac{1}{2}\,\sigma\,{t}^{-1}-\sigma\,t +2\,\sigma\,{t}^{3}+O \left( {t}^{4}
 \right) ), \\
f|_{\Gamma_{\sigma, 4}} &=& (\sigma\,t - \frac{1}{4}\,{\sigma}^{2}{t}^{2}-2\,\sigma\,{t}^{3}+2\,{\sigma}^{2}
{t}^{4}+O \left( {t}^{5} \right) ).
\end{array}$$
It follows that
\begin{eqnarray*}
\lambda_{\sigma, 1} &=& \lambda_{\sigma, 2}  = \ \lambda_{\sigma, 3}  \ = \ +\infty, \quad \lambda_{\sigma, 4}  \ = \ 0.
\end{eqnarray*}
In light of Theorem~\ref{Theorem41}, $f$ is bounded from below. Note that
\begin{eqnarray*}
f|_{\Sigma(f, \mathbb{R}^2)}  & \equiv & 1 \ > \ 0 \ = \ \min_{k = 1, \ldots, 4} \lambda_{\sigma, k}.
\end{eqnarray*}
Hence, by Theorem~\ref{Theorem44}, $f$ does not attain its infimum. In view of Theorem~\ref{Theorem43}, we have
\begin{eqnarray*}
\inf_{(x, y) \in \mathbb{R}^2} f(x, y) &=& 0.
\end{eqnarray*}
Furthermore, by Theorem~\ref{Theorem46}, $f$ is not coercive.
}\end{example}

\begin{example}\label{ex54}{\rm
Let $f(x,y):=\min\{x+y,1\}$  and let $S:=\R^2_+.$
Then the function $f$ is semi-algebraic and $\mathrm{(QC)}$ holds.
Note that the function $f$ is continuous and concave.
Then it follows from \cite[Proposition~7]{Ioffe1984} that we have
\begin{eqnarray*}
\partial f(x,y) =   \left\{
                    \begin{array}{ll}
                      \{(0,0)\} & \text{if } x+y>1, \\
                     \{(0,0), (1,1)\}, & \text{if } x+y=1, \\
                      \{(1,1)\}, & \text{if } x+y<1.
                    \end{array}\right.
\end{eqnarray*}
Moreover, by a simple calculation, we see that
\begin{eqnarray*}
N((x,y);S) =  \left\{
                    \begin{array}{ll}
                     -\R^2_+ & \text{if } (x,y)=(0,0), \\
                     \{0\}\times(-\R_+), & \text{if } x>0, \ y=0, \\
                      -\R_+\times\{0\}, & \text{if } x=0, \ y>0, \\
                      \{(0,0)\}, & \text{if } x>0, \ y>0,
                    \end{array}\right.
\end{eqnarray*}
and so,
\begin{eqnarray*}
\Gamma(f, S) &=&  \{ (x,y)\in S \ | \ x=y, \ x+y<1\}\cup\{ (x,y)\in S \ | \  \ x+y\geq1\}\\
&&\cup \{ (x,y)\in S \ | \  \ x=0, \  y>0\}\cup \{ (x,y)\in S \ | \  \ x>0, \  y=0\}\cup \{(0,0)\}.
\end{eqnarray*}
Note that
\begin{eqnarray*}
\Sigma(f,S)=  \{ (x,y)\in S \ | \  \ x+y\geq1\}\cup \{(0,0)\}
\end{eqnarray*}

Now, for $R>1,$ let $\Gamma_1:=\Gamma(f,S)\setminus \mathbb{B}_{R}.$
Then we have $f|_{\Gamma_1}\equiv1,$ and so $\lambda_1=\lim_{t\to\infty}f|_{\Gamma_1}=1.$
Thus,
\begin{eqnarray*}
\min_{(x,y) \in \Sigma(f, S) } f(x,y)=0<1=\lambda_1.
\end{eqnarray*}
By Theorem~\ref{Theorem45}, the set of all optimal solutions of $f$ on $S$ is nonempty compact, that is $\{(0,0)\}.$
Moreover, it follows from Theorem~\ref{Theorem46} that the function $f$ is not coercive on $S.$
}\end{example}

\begin{example}{\rm
Let $S:=\R^2_+.$ Consider the following function from $\R^2$ to $\R:$
\begin{eqnarray*}
f(x,y) =  \left\{
                    \begin{array}{ll}
                      0, & \text{if } (x,y)\in A, \\
                      1, & \text{if } (x,y)\notin A,
                    \end{array}\right.
\end{eqnarray*}
where $A:=\{(x,y) \ | \ x\in\R, \ y\geq0\}.$
Note that the function $f$ is lower semi-continuous and semi-algebraic, but not local Lipschitz.
Note also that Normal cone to $S$ at $x$ is same in Example~\ref{ex54}.
So, by a direct calculation, we see that
\begin{eqnarray*}
\partial f(x,y) =  \left\{
                    \begin{array}{ll}
                     \{0\}\times(-\R_+)& \text{if } x\geq0, \ y=0, \\
                     \{(0,0)\}, & \text{if } x\geq0, \ y>0.
                    \end{array}\right.
\end{eqnarray*}
and $\Gamma(f,S)=\R_+^2.$
Let $\Gamma_1:=\Gamma(f,S)\setminus \mathbb{B}_{R}$ for $R>0.$
Then we have $f|_{\Gamma_1}\equiv0,$ and so, $\lambda_1=\lim_{t\to\infty}f|_{\Gamma_1}=0.$
Note that $\Sigma(f,S)=\R_+^2.$ Then we see that
\begin{eqnarray*}
\min_{(x,y) \in \Sigma(f, S) } f(x,y)=0=\lambda_1.
\end{eqnarray*}
So, it follows from Theorem~\ref{Theorem44}, $f$ attains its infimum on $S,$ that is, $0.$
Moreover, by Theorem~\ref{Theorem45}, we see that the set of all optimal solutions of $f$ on $S$ is nonempty, but not compact.
}\end{example}


\end{document}